\documentclass[reqno,11pt,a4paper]{amsart}
\usepackage[T1]{fontenc}
\usepackage{graphicx,mathtools,amssymb,amsthm,thmtools}
\usepackage{xcolor,babel,amsmath,enumitem,bm}
\usepackage{appendix}
\usepackage{tikz-cd,caption}
\usetikzlibrary{patterns, arrows.meta, positioning}
\usepackage[numbers]{natbib}
\newtheorem{theorem}{Theorem}[section]
\newtheorem{lemma}{Lemma}[section]
\newtheorem{definition}{Definition}[section]

\newtheorem{remark}{Remark}[section]
\numberwithin{equation}{section}
\usepackage[%
colorlinks=true,  
linkcolor=blue,  
citecolor=blue,  
urlcolor=black,   
anchorcolor=black,
]{hyperref}
\makeatletter
\@ifundefined{theHtheorem}{}{%
}
\@ifundefined{theHlemma}{}{%
}
\@ifundefined{theHremark}{}{%
}
\@ifundefined{theHproposition}{}{%
}

\makeatother
\usepackage{mathrsfs}
\allowdisplaybreaks[4]
\begin{document}
	\title[ Sharp well-posedness and ill-posedness of the Camassa-Holm equation]
	{Sharp well-posedness and ill-posedness of the Camassa-Holm equation in critical Triebel-Lizorkin spaces}
	\author{Qianyuan Zhang}
	\address{Qianyuan Zhang\newline
		School of Mathematics and Statistics\\
		Huazhong University of Science and Technology, Wuhan 430074,  China}
	\email{qianyuanzhang@hust.edu.cn}
	
	\author{Kai Yan\textsuperscript{*}}
	\thanks{\noindent $^*$Corresponding author.}
	\address{Kai Yan (Corresponding author)\newline
		School of Mathematics and Statistics\\
		Huazhong University of Science and Technology, Wuhan 430074,  China}
	\email{kaiyan@hust.edu.cn}
	
	\begin{abstract}
		This paper is devoted to the sharp well-posedness and ill-posedness of the Cauchy problem for the Camassa-Holm (CH) equation in critical Triebel-Lizorkin spaces $F^{1+\frac{1}{p}}_{p,q}(\mathbb{R})$ with $(p,q)\in[1,\infty)\times[1,\infty]$ or $p=q=\infty$. On the one hand, we establish the local well-posedness in the sense of Hadamard in $F^2_{1,q}(\mathbb{R})$ for $1\leq q<\infty$ via Lagrangian coordinate transformation. On the other hand, by means of smooth atomic decomposition, strong ill-posedness is then proved in $F^{1+\frac{1}{p}}_{p,q}(\mathbb{R})$ with $(p,q)\in(1,\infty)\times[1,\infty]$ or $p=q=\infty$ in the sense of norm inflation, which in particular yields the ill-posedness of CH in critical Sobolev spaces $W^{1+\frac{1}{p},p}(\mathbb{R})$ with $1<p<\infty$,  and provides a new perspective on the ill-posedness of CH in $H^{\frac{3}{2}}(\mathbb{R})$.
	\end{abstract}

	\maketitle
	
	
	\vskip 0.2cm
	
	
	\setcounter{equation}{0}

	\section{Introduction}
	Over the past decades, a large amount of literature was devoted to the following celebrated Camassa-Holm (CH) equation: 
	\begin{equation}\label{CH}
		u_t-u_{txx}+3uu_x=2u_xu_{xx}+uu_{xxx},
	\end{equation}
	which is proposed as a bi-Hamiltonian equation \cite{MR636470} and models the unidirectional propagation of shallow water waves over a flat bottom \cite{MR1234453} (see also a rigorous justification in shallow water \cite{MR2481064}). Here $u(t,x)$ stands for the fluid velocity at time $t$ and in the spatial $x$ direction. The CH equation \eqref{CH} is also recognized as a model for the propagation of axially symmetric waves in hyperelastic rods \cite{MR1606738}. It has a bi-Hamiltonian structure and is completely integrable \cite{MR1234453}. Its solitary waves are peaked solitons (peakons) \cite{MR1234453,MR2318158} (see also \cite{MR2257390,MR2753609} for the waves of great height in irrotational water waves), and they are orbitally stable \cite{MR1737505}. Moreover, for the Cauchy problem of CH equation \eqref{CH}, it has both global strong solutions \cite{MR1775353,MR1631589} and blow-up solutions within finite time \cite{MR1775353,MR1631589,MR1668586} which is featured as wave breaking \cite{MR1668586} (namely, the wave remains bounded while its slope becomes unbounded in finite time \cite{MR483954}). In comparison with the KdV equation, the advantage of the CH equation \eqref{CH} lies in the fact that the CH equation has peakons and models wave breaking.
	
	Now, let us rewrite the Cauchy problem for \eqref{CH} as its nonlocal form:
	\begin{equation}\label{CH-transport-equation}
		\begin{cases}
			\partial_tu+u\partial_xu=P(D)\big(u^2+\frac{1}{2}(\partial_xu)^2\big),&t>0, \quad x\in\mathbb{R},\\
			u(0,x)=u_0(x),& x\in\mathbb{R}
		\end{cases}
	\end{equation}
	with $P(D)=-\partial_x(1-\partial^2_x)^{-1}$. The well-posed issue of the Cauchy problem \eqref{CH-transport-equation} has been extensively studied. Li and Olver \cite{MR1741872} (see also \cite{MR1851854}) proved that the Cauchy problem \eqref{CH-transport-equation} is locally well-posed with the initial data $u_0\in H^s(\mathbb{R})$ with $s>\frac{3}{2}$. The index $s=\frac{3}{2}$ is critical, since \eqref{CH-transport-equation} is ill-posed  in $H^s(\mathbb{R})$ for $s<\frac{3}{2}$ (see \cite{MR2244592}). Moreover, The ill-posedness in the critical space $H^{\frac{3}{2}}(\mathbb{R})$ was established in \cite{MR3906228}  by proving norm inflation in the general Besov spaces $B^{1+1/p}_{p,q}(\mathbb{R})$ for $1\leq p\leq \infty$ and $1<q\leq\infty$. With regard to the case in the general critical Sobolev spaces $W^{1+1/p,p}(\mathbb{R})$ for $1\leq p\leq \infty$, it was shown in \cite{MR4635983} that \eqref{CH-transport-equation} is well-posed in $W^{2,1}(\mathbb{R})$. Ill-posedness in $W^{1,\infty}(\mathbb{R})$ was proved in \cite{MR4388135} by showing the failure of continuous dependence. Guo and Ye \cite{MR4635983} claimed that the CH equation is ill-posed in $W^{1+1/p, p}(\mathbb{R})$ with $1<p<\infty$. On the other hand, Danchin \cite{MR1827098,MR1990847} investigated the case in Besov spaces, establishing local well-posedness in $B^s_{p,q}(\mathbb{R})$ for $s>\max(1+\frac{1}{p},\frac{3}{2})$ with $1\leq p,q\leq \infty,$ as well as for the endpoint case $s=1+\frac{1}{p}$ with $1\leq p \leq 2$ and $q=1$. The well-posedness in the critical Besov space $B^{1+1/p}_{p,1}(\mathbb{R})$ for $1\leq p<\infty$ was later proved by Ye, Yin and Guo \cite{MR4595933}, whereas the ill-posedness in the limiting space $B^1_{\infty,1}(\mathbb{R})$ was established in \cite{MR4413296}.
	
	As is well known, the Triebel-Lizorkin spaces $F^s_{p,q}$ provide a unified framework covering many classical function spaces arising in the theory of partial differential equations. Notable identifications include (see the details in \cite{MR2768550,MR781540}):
	\begin{itemize}[left=0pt]
		\item The Sobolev spaces $W^{s,p}=F^s_{p,2}$ for $s\in\mathbb{R},1<p<\infty.$ In particular, $H^s=F^s_{2,2}$.
		\item The H\"{o}lder-Zygmund spaces $\mathscr{C}^s=F^s_{\infty,\infty}$ for $s>0. $
		\item The local Hardy spaces $h_p=F^0_{p,2}$ for $0<p<\infty$. 
		\item The local BMO spaces $bmo=F^0_{\infty,2}$. 
		\item The Besov spaces: $B^s_{p,\min(p,q)} \hookrightarrow F^s_{p,q}\hookrightarrow B^s_{p,\max(p,q)}$ $(s\in\mathbb{R},(p,q)\in[1,\infty)\times [1,\infty]$ or $p=q=\infty$). In particular, $F^s_{p,p}=B^s_{p,p}$ $(s\in\mathbb{R},1\leq p\leq \infty)$.
	\end{itemize}
	Recently, Zhang and Yan \cite{zhang2026commutatorestimatesapplicationstransporttype} proved the local well-posedness for \eqref{CH-transport-equation} in subcritical Triebel-Lizorkin spaces $F^s_{p,q}(\mathbb{R})$ for $s>\max(1+\frac{1}{p},\frac{3}{2})$ with $(p,q)\in[1,\infty)\times[1,\infty]$ or $p=q=\infty$. To the best of our knowledge, the critical case in the Triebel-Lizorkin scale remains completely open. The goal of this paper is therefore to study the well-posedness of the CH equation in the critical Triebel-Lizorkin spaces $F^{1+1/p}_{p,q}(\mathbb{R})$ with $(p,q)\in[1,\infty)\times[1,\infty]$ or $p=q=\infty$. More precisely, we prove that \eqref{CH-transport-equation} is local well-posed in the sense of Hadamard in $F^2_{1,q}(\mathbb{R})$ for $1\leq q<\infty$, and on the other hand, it is ill-posed in $F^{1+1/p}_{p,q}(\mathbb{R})$ with $(p,q)\in(1,\infty)\times[1,\infty]$ or $p=q=\infty$ in the sense of norm inflation.
	
	 When discussing the well-posedness of \eqref{CH-transport-equation} in the critical space $F^{2}_{1,q}(\mathbb{R})$, the main difficulty is to prove the uniqueness, since the following Moser-type inequality fails to hold:
	\[ \|fg\|_{F^0_{1,q}(\mathbb{R})}\leq C\|f\|_{F^0_{1,q}(\mathbb{R})}\|g\|_{F^1_{1,q}(\mathbb{R})}. \]
	To circumvent the use of such Moser-type estimates, and motivated by the approach in \cite{MR4595933}, we employ Lagrangian coordinate transformations to investigate the uniqueness of solutions to the CH equation. Let us first consider the following abstract equation:
	\begin{equation}\label{general-abstract-equation}
		\begin{cases}
			\partial_tu+A(u)\partial_xu=G(u),& t>0,\quad x\in\mathbb{R},\\
			u(0,x)=u_0(x),&x\in\mathbb{R},
		\end{cases}
	\end{equation}
	where $A(u)$ is a polynomial of $u$ and $G$ is called a \textquoteleft good operator' such that for any $\phi\in C^\infty_c(\mathbb{R})$ and $\varepsilon>0$ small enough, the following fact holds
	\[ \text{if}\qquad u_n\phi\to u\phi\qquad\text{in}\qquad F^{2-\varepsilon}_{1,q}(\mathbb{R}),\qquad\text{then}\qquad \langle G(u_n),\phi \rangle\to\langle G(u),\phi\rangle. \] 
	Let $X(t,\alpha)$ be the flow of $A(u)$, i.e., the solution to the following ordinary differential equation:
	\begin{equation}\label{ODE}
		\begin{cases}
			\partial_tX(t,\alpha)=A\big(u(t,X(t,\alpha))\big),&\quad t>0,\alpha\in\mathbb{R},\\
			X(0,\alpha)=\alpha,&\quad\alpha\in\mathbb{R}.
		\end{cases}
	\end{equation}
	Set $U(t,\alpha)=u(t,X(t,\alpha))$, then we have 
	\begin{equation}\label{Lagrangian}
		\begin{cases}
			U_t=G(u(t,X(t,\alpha)))\triangleq\widetilde{G}(U,X),&\quad t>0,\alpha\in\mathbb{R},\\
			U(0,\alpha)=U_0(\alpha)=u_0(\alpha),&\quad\alpha\in\mathbb{R}.
		\end{cases}
	\end{equation}
	Now, we state our first main theorem of this paper as follows.
	\begin{theorem}\label{thm:abstract-Local-well-posedness}
		Let $1\leq q< \infty$ and $k\in\mathbb{N}^+$. Suppose that $u_0\in F^{2}_{1,q}(\mathbb{R})$ and $G$ is a \textquoteleft good operator\textquoteright . If the following conditions for $G,\widetilde{G}$ hold
		\begin{align}
			\|G(u)\|_{F^2_{1,q}}
			&\leq C \big(\|u\|^{k+1}_{F^2_{1,q}}+1\big),\label{eq:existence}\\
			\|\widetilde{G}(U,X)-\widetilde{G}(\bar{U},\bar{X})\|_{W^{1,\infty}\cap W^{1,1}}
			&\leq C \big(\|U-\bar{U}\|_{W^{1,\infty}\cap W^{1,1}}+\|X-\bar{X}\|_{W^{1,\infty}\cap W^{1,1}}\big),\label{eq:uniqueness}\\
			\|G(u)-G(\bar{u})\|_{F^2_{1,q}}
			&\leq C\|u-\bar{u}\|_{F^2_{1,q}}.\label{eq:continuous}
		\end{align} 
		Then, there exists a time $T>0$ such that
		\begin{enumerate}[label=(\arabic*),leftmargin=*]
			\item Existence: If \eqref{eq:existence} holds, then \eqref{general-abstract-equation} has a solution $u\in E_q(T)\triangleq C([0,T];F^2_{1,q}(\mathbb{R}))\cap C^1([0,T];F^1_{1,q}(\mathbb{R}))$;
			\item  Uniqueness: If \eqref{eq:existence} and \eqref{eq:uniqueness} hold, then the solution of \eqref{general-abstract-equation} is unique;
			\item Continuous dependence: If \eqref{eq:existence}-\eqref{eq:continuous} hold, then the solution map is continuous from any bounded subset of $F^2_{1,q}(\mathbb{R})$ to $C([0,T];F^2_{1,q}(\mathbb{R}))$.
		\end{enumerate}
	\end{theorem}
	
	With this abstract Theorem \ref{thm:abstract-Local-well-posedness} in hand, we immediately get the following result.
	
	\begin{theorem}\label{thm:CH-wellposedness}
		Let $u_0\in F^{2}_{1,q}(\mathbb{R})$ with $1\leq q< \infty$. Then there exists a time $T>0$ such that the Camassa-Holm equation \eqref{CH-transport-equation} with initial data $u_0$ is locally well-posed in the sense of Hadamard.
	\end{theorem}
	
	\begin{remark}\label{ref:remark1}
		We point out that the well-posedness result for the CH equation was established in \cite{zhang2026commutatorestimatesapplicationstransporttype} in the subcritical Triebel-Lizorkin spaces $F^s_{p,q}(\mathbb{R})$ for $s>\max(1+\frac{1}{p},\frac{3}{2})$, $(p,q)\in[1,\infty)\times[1,\infty]$ or $p=q=\infty$. Nevertheless, by applying the Lagrangian transformation used in the uniqueness proof of Theorem \ref{thm:abstract-Local-well-posedness}, one readily gets the local well-posedness for \eqref{CH-transport-equation} in $F^s_{p,q}(\mathbb{R})$ as $s>1+\frac{1}{p}$ with $(p,q)\in[1,\infty)\times[1,\infty]$, thereby the condition $s>\frac{3}{2}$ can be removed and hence it improves the corresponding result in \cite{zhang2026commutatorestimatesapplicationstransporttype}.
	\end{remark}
	
	Recall that $F^1_{\infty,\infty}(\mathbb{R})=B^1_{\infty,\infty}(\mathbb{R})$. It has already been shown in \cite{MR3906228} that the CH equation is ill-posed via norm inflation in $F^1_{\infty,\infty}(\mathbb{R})$. For the remaining range $1<p<\infty$ and $1\leq q\leq\infty$, we can also prove that \eqref{CH-transport-equation} exhibits norm inflation in $F^{1+1/p}_{p,q}(\mathbb{R})$ (see Theorem \ref{thm:CH-illposedness} below). Compared with the Besov space result in \cite{MR3906228}, the Triebel-Lizorkin setting presents a more delicate problem. The main difficulty stems from the contrasting structures of the two norms. The Besov norm is an $l^q(L^p)$ norm: one first computes the $L^p$ norm of each dyadic frequency block, and then the $l^q$ norm over the block index. Because the $l^q$-summation is taken after the $L^p$ norms, the contributions from different dyadic blocks do not interfere substantially with one another. This is why the initial data in \cite{MR3906228} can be built as a simple superposition $\sum_ka_kh_k$ with the $h_k$ having disjoint Fourier supports. Up to equivalence, the Besov norm then factorises into a sequence norm of the coefficients $a_k$ alone. The Triebel-Lizorkin norm, in contrast, is an $L^p(l^q)$ norm: the $l^q$ combination of the frequency blocks is taken pointwise in $x$ before the outer $L^p$ integration, so that the contributions from all dyadic scales are entangled at every spatial point. In this topology, the block-by-block construction that works so well for Besov spaces offers no direct control. To overcome this difficulty, we adopt an atomic approach, using the smooth atomic decomposition for Triebel-Lizorkin spaces (see Lemmas \ref{lem:F-atoms} and \ref{lem:atom-1}). This atomic decomposition is the key to making the construction of counterexamples tractable. More precisely, we prove the following result.
	
	\begin{theorem}\label{thm:CH-illposedness}
		Let $1< p<\infty$ and $1\leq q\leq \infty$. $\forall\,\varepsilon>0$, there exists $u_0\in H^\infty(\mathbb{R})$, real-valued, such that the following hold:
		\begin{enumerate}
			\item $\|u_0\|_{F^{1+\frac{1}{p}}_{p,q}(\mathbb{R})}\leq \varepsilon$;
			\item There is a unique solution $u\in C([0,T);H^\infty(\mathbb{R}))$ to the Cauchy problem \eqref{CH-transport-equation} with a maximal lifespan $T<\varepsilon$;
			\item  $\limsup_{t\to T^-}\|u(t)\|_{F^{1+\frac{1}{p}}_{p,q}(\mathbb{R})}\geq \limsup_{t\to T^-}\|u(t)\|_{F^1_{\infty,\infty}(\mathbb{R})}=\infty$.
		\end{enumerate}
	\end{theorem}
	
	\begin{remark}
		Taking $q=2$, Theorem \ref{thm:CH-illposedness} implies the ill-posedness of \eqref{CH-transport-equation} in critical Sobolev spaces $W^{1+\frac{1}{p},p}(\mathbb{R})$ with $1<p<\infty$, and provides a new perspective on the ill-posedness for \eqref{CH-transport-equation} in $H^\frac{3}{2}(\mathbb{R})$, a long-standing open problem for the CH equation.
	\end{remark}
	
	Let us summarize the complete well-posedness/ill-posdeness results of \eqref{CH-transport-equation} in the Triebel-Lizorkin spaces, which can be seen clearly from the Table 1.
	\begin{table}[h]
		\centering
		\small 
		\begin{tabular}{c|c|c}
			\hline
			\textup{References} & \textup{Range} & \textup{Results} \\ \hline
			\textup{
			Remark \ref{ref:remark1}} and
			\cite{zhang2026commutatorestimatesapplicationstransporttype}
			& $s>1+\frac{1}{p},1\leq p<\infty,1\leq q\leq\infty$ & \textup{Local well-posedness} \\ \hline
			\textup{Theorem \ref{thm:CH-wellposedness}}&$s=2,p=1,1\leq q<\infty$& \textup{Local well-posedness} \\ \hline
			\textup{Theorem \ref{thm:CH-illposedness}}&$s=1+\frac{1}{p},1< p<\infty,1\leq q\leq\infty$ & \textup{Norm inflation} \\ \hline
			\textup{\cite{MR3906228}} & $s=1,p=q=\infty$ & \textup{Norm inflation} \\ \hline
		\end{tabular}
		\caption{Well/Ill-posedness of \eqref{CH-transport-equation} in $F^s_{p,q}(\mathbb{R})$}
	\end{table}

		\begin{remark}
		It is worth pointing out that the analogical results with Theorems \ref{thm:CH-wellposedness} and \ref{thm:CH-illposedness} also hold true for the other related shallow water models, such as the Degasperis-Procesi equation, $b$-family equation, Novikov equation and two-component Camassa-Holm system.
	\end{remark}

	The remainder of the paper is organized as follows. In Section 2, we recall some facts on Triebel-Lizorkin spaces and the theory of transport equations. In Section 3, we establish a well-posedness result (Theorem \ref{thm:abstract-Local-well-posedness}) for a general abstract equation \eqref{general-abstract-equation} in critical Triebel-Lizorkin spaces, and then apply it to obtain the well-posedness for the CH equation (Theorem \ref{thm:CH-wellposedness}). In Section 4, we introduce the atomic decomposition to prove the ill-posedness result Theorem \ref{thm:CH-illposedness}.
	
	\section{Preliminaries}\label{section:2}
	In this section, we provide some basic facts concerning the Littlewood-Paley decomposition, Triebel-Lizorkin spaces and the transport equation theory. Let $\mathscr{B}=\{\xi\in\mathbb{R}\colon|\xi|\leq \frac{4}{3} \}$ and $\mathscr{C}=\{\xi\in\mathbb{R}\colon\frac{3}{4}\leq |\xi|\leq \frac{8}{3} \}$, $(\varphi,\chi)$ be a couple of smooth functions valued in $[0,1]$, such that $\varphi$ is support in  $\mathscr{C}$, $\chi$ is supported in $\mathscr{B}$ and 
	\[ \forall \xi \in \mathbb{R},\quad \chi(\xi)+\sum_{j\in\mathbb{N}}\varphi(2^{-j}\xi)=1. \]
	We denote $\varphi_j(\xi)=\varphi(2^{-j}\xi),h=\mathscr{F}^{-1}\varphi$ and $\tilde{h}=\mathscr{F}^{-1}\chi$. For $f\in\mathscr{S}'(\mathbb{R})$, one can define nonhomogeneous dyadic blocks as follow:
	\begin{equation}
		\Delta_jf\triangleq0,\quad\text{if}\quad j\leq -2,
	\end{equation}
	\begin{equation}
		\Delta_{-1}f\triangleq\chi(D)f=\tilde{h}*f,
	\end{equation}
	\begin{equation}\label{def:Delta_j}
		\Delta_jf\triangleq\varphi(2^{-j}D)f=2^{jd}\int_{\mathbb{R}}h(2^jy)f(x-y)dy,\quad\text{if}\quad  j\geq 0,
	\end{equation}
	\begin{equation}\label{def:S_j}
		S_jf\triangleq\sum_{k\leq j-1}\Delta_kf=\chi(2^{-j}D)f=2^{jd}\int_{\mathbb{R}}\tilde{h}(2^jy)f(x-y)dy.
	\end{equation}
	Let $s\in\mathbb{R},(p,q)\in [1,\infty)\times[1,\infty]$ or $p=q=\infty$. The Triebel-Lizorkin space $F^s_{p,q}(\mathbb{R})$\footnote{For the endpoint case $p=q=\infty$, $F^s_{\infty,\infty}$ is equipped with the norm
		\[ \|f\|_{F^s_{\infty,\infty}}=\sup_{j\in \mathbb{Z}}2^{js}\|\Delta_jf\|_{L^\infty}, \]
		which coincides with the Besov spaces $B^s_{\infty,\infty}$.} ($F^s_{p,q}$ for short) is defined by \cite{MR781540}
	\begin{equation}\label{def:norm_nonhomogeneous T-L}
		F^s_{p,q}(\mathbb{R})=\Big\{f\in \mathscr{S}'(\mathbb{R})\colon \|f\|_{F^s_{p,q}}=\Big\|\big\|2^{js}|\Delta_jf|\big\|_{l^q(j\geq -1)}\Big\|_{L^p(\mathbb{R})}<\infty\Big\}.
	\end{equation}
	
	Let us give some classical properties for the Triebel-Lizorkin as follows. 
	\begin{lemma}\label{lem:Triebel-Lizorkin-properties}
	    \textup{(\cite{MR1419319,MR781540})} The following properties hold:
		\begin{enumerate}[label={(\roman*)}]
		\item Embedding: Let $\varepsilon>0$ and suppose $q_0<q_1$, $p_0<p_1 $ and $s_0-\frac{1}{p_0}=s_1-\frac{1}{p_1}$. Then
		\[ F^s_{p,q_0}(\mathbb{R})\hookrightarrow F^s_{p,q_1}(\mathbb{R}),\qquad F^{s+\varepsilon}_{p,q}(\mathbb{R})\hookrightarrow F^s_{p,q}(\mathbb{R}),\qquad F^{s_0}_{p_0,\infty}(\mathbb{R})\hookrightarrow F^{s_1}_{p_1,q}(\mathbb{R}). \]
		\item Compact embedding: If $s_1<s_2$, then the embedding $F^{s_2}_{p,q}\hookrightarrow F^{s_1}_{p,q}$ is locally compact.
		\item Algebraic properties: For $s>0$, $F^s_{p,q}\cap L^\infty$ is an algebra. Moreover, $F^s_{p,q}(\mathbb{R}) $ is an algebra $\Longleftrightarrow$ $F^s_{p,q}(\mathbb{R})\hookrightarrow L^\infty(\mathbb{R})\Longleftrightarrow s>\frac{1}{p}$ (or $s\geq 1$ and $p=1$).
		\item Fatou property: If the sequence $\{f_k\}_{k\in\mathbb{N}^+}$ is uniformly bounded in $F^s_{p,q}$ and converges weakly in $\mathscr{S}'$ to $f$, then $f\in F^s_{p,q}$ and $\|f\|_{F^s_{p,q}}\leq \liminf\limits_{k\to\infty}\|f_k\|_{F^s_{p,q}}$.
		\item Complex interpolation: Let $1\leq p_0,q_0\leq \infty,1\leq p_1,q_1\leq \infty,0<\theta<1$ and
		\[ \frac{1}{q}=\frac{1-\theta}{q_0}+\frac{\theta}{q_1},\quad\frac{1}{p}=\frac{1-\theta}{p_0}+\frac{\theta}{p_1},\quad  s=(1-\theta)s_0+\theta s_1. \] 
		Then we have
		\[ \|f\|_{F^{s}_{p,q}}\leq\|f\|_{F^{s_0}_{p_0,q_0}}^{1-\theta}\|f\|_{F^{s_1}_{p_1,q_1}}^\theta,\quad \forall \ f\in F^{s_0}_{p_0,q_0}\cap F^{s_1}_{p_1,q_1}.\] 
		\item Lifting property: Let $I_{\sigma}f(x)\triangleq\mathscr{F}^{-1}[(1+|\xi|^2)^{\frac{\sigma}{2}}\mathscr{F}f(\xi)](x),f\in \mathscr{S}'(\mathbb{R}),\sigma\in\mathbb{R}.$
		Then $I_{\sigma}$ maps $F^s_{p,q}$ isomorphically onto $F^{s-\sigma}_{p,q}$. Furthermore,
		\[ \sum_{|\alpha|\leq m}\|D^\alpha f\|_{F^{s-m}_{p,q}(\mathbb{R})} \qquad\text{and}\qquad \|f\|_{F^{s-m}_{p,q}(\mathbb{R})}+\Big\|\frac{\partial^mf}{\partial x^m}\Big\|_{F^{s-m}_{p,q}(\mathbb{R})}\]
		are equivalent norms on $F^s_{p,q}(\mathbb{R})$.
		\end{enumerate}
	\end{lemma}
	
	We now recall the theory of transport equation in the framework of Triebel-Lizorkin spaces.

	\begin{lemma}\label{lem:a-priori-estimates}
		\textup{(\cite{zhang2026commutatorestimatesapplicationstransporttype})}(A priori estimates for transport equation) Let $(p,q)\in[1,\infty)\times[1,\infty]$ or $p=q=\infty$. Assume that $s>-\min\big(\frac{1}{p},1-\frac{1}{p}\big)$, and $\partial_xv$ belongs to $L^1(0,T;F^{s-1}_{p,q}(\mathbb{R}))$ if $s>1+\frac{1}{p}$ (or $s=1+\frac{1}{p}$ and $p=1$), or to $L^1(0,T;F^{\frac{1}{p}}_{p,\infty}(\mathbb{R})\cap L^\infty(\mathbb{R}))$ otherwise. Suppose that $f_0\in F^s_{p,q}(\mathbb{R})$,
		 $g \in L^1(0,T;F^s_{p,q}(\mathbb{R}))$ and that 
		$f\in L^\infty(0,T;F^s_{p,q}(\mathbb{R}))$ solves the following transport equation:
		\begin{equation}\label{eq:transport}\tag{T}
			\begin{cases}
				\partial_tf+v\partial_x f=g,\\
				f|_{t=0}=f_0.			
			\end{cases}
		\end{equation}
		Then there exists a constant $C$ depending only on $p,q$ and $s$, such that the following estimates hold:
		\begin{equation}
			\|f(t)\|_{F^s_{p,q}}\leq \|f_0\|_{F^s_{p,q}}+\int_{0}^t\|g(\tau)\|_{F^s_{p,q}}d\tau+C\int_{0}^tZ(\tau)\|f(\tau)\|_{F^s_{p,q}}d\tau,
		\end{equation}
		or hence, 
		\begin{equation}\label{eq:a-priori-estimates}
			\|f(t)\|_{F^s_{p,q}}\leq
			e^{C\int_{0}^tZ(\tau)d\tau} \Big(\|f_0\|_{F^s_{p,q}}+\int_0^te^{-C\int_{0}^\tau Z(\tau')d\tau'}\|g(\tau)\|_{F^s_{p,q}}d\tau\Big)
		\end{equation}
		with
		\begin{equation*}
			Z(t)=\begin{cases}
				\|\partial_x v(t)\|_{F^{\frac{1}{p}}_{p,\infty}\cap L^\infty},\ &\text{if}\quad s<1+\frac{1}{p},\\
				\|\partial_xv(t)\|_{F^{s-1}_{p,q}},\ &\text{if}\quad s>1+\frac{1}{p}\ \text{or}\ s=1+\frac{1}{p}\ \text{and}\quad p=1.
			\end{cases}
		\end{equation*}
		If moreover $f=v$, then for all $s>0$ estimates \eqref{eq:a-priori-estimates} hold with $Z(t)=\|\partial_xv(t)\|_{L^\infty}.  $ 
	\end{lemma}
	\begin{lemma}\label{lem:solve-transport}
		\textup{(\cite{zhang2026commutatorestimatesapplicationstransporttype})}(Well-posedness for transport equation) Let $p,q,s,f_0$ and $g$ be as in the statement of Lemma \ref{lem:a-priori-estimates}. Assume that $v\in L^\rho(0,T;F^{-M}_{\infty,\infty}(\mathbb{R}))$ for some $\rho>1$ and $M>0$, and such that $\partial_xv\in L^1(0,T;F^{\frac{1}{p}}_{p,\infty}(\mathbb{R})\cap L^\infty(\mathbb{R}))$ if $s<1+\frac{1}{p}$, and $\partial_x v\in L^1(0,T;F^{s-1}_{p,q}(\mathbb{R}))$ if $s>1+\frac{1}{p}$ or $s=2$ with $p=1$. Then the transport equation \eqref{eq:transport} has a unique solution $f\in L^\infty([0,T];F^s_{p,q}(\mathbb{R}))\bigcap(\cap_{s'<s}C([0,T];F^{s'}_{p,1}(\mathbb{R}))) $ and the estimates in Lemma \ref{lem:a-priori-estimates} hold true. If moreover $q<\infty$, then $f\in C([0,T];F^s_{p,q}(\mathbb{R}))$. Furthermore, the data-to-solution map $f_0\mapsto f$ is continuous from $F^s_{p,q}(\mathbb{R})$ into $L^\infty ([0,T];F^{s}_{p,q}(\mathbb{R}))\cap C([0,T];F^{s'}_{p,1}(\mathbb{R}))$ for every $s'<s$ if $q=\infty$, and into $C ([0,T];F^{s}_{p,q}(\mathbb{R}))$ if $q<\infty$.
	\end{lemma}

	\section{Local well-posedness}
	This section is devoted to the local well-posedness results (Theorems \ref{thm:abstract-Local-well-posedness} and \ref{thm:CH-wellposedness}). As a first step, we prove the following lemma which will play a crucial role in the subsequent argument for continuous dependence.
	
	\begin{lemma}\label{lem:continuity-transport}
		Denote $\bar{\mathbb{N}}=\mathbb{N}\cup\{\infty\}$. Let $1\leq q<\infty$ and $(f^n)_{n\in\bar{\mathbb{N}}}$ be a sequence of functions belonging to $C([0,T];F^{1}_{1,q}(\mathbb{R}))$. Assume that $f^n$ is the solution to
		\begin{equation*}
			\begin{cases}
				\partial_tf^n+v^n\partial_xf^n=g,\\
				f^n|_{t=0}=f_0,
			\end{cases}
		\end{equation*}
		with $f_0\in F^{1}_{1,q}(\mathbb{R}),g\in L^1(0,T;F^{1}_{1,q}(\mathbb{R}))$ and that, for some $a\in L^1(0,T)$,
		\begin{equation*}
			\sup_{n\in\bar{\mathbb{N}}}\|\partial_xv^n(t)\|_{F^{1}_{1,q}}\leq a(t).
		\end{equation*}
		If in addition $v^n$ tends to $v^\infty$ in $L^1(0,T;F^1_{1,q}(\mathbb{R}))$, then $f^n$ tends to $f^\infty$ in $C([0,T];F^{1}_{1,q}(\mathbb{R}))$.
	\end{lemma}
	\begin{proof}
		Let $w^n\triangleq f^n-f^\infty$. We have
		\begin{equation*}
			\partial_tw^n+v^n\partial_xw^n=(v^\infty-v^n)\partial_xf^\infty,\qquad w^n|_{t=0}=0.
		\end{equation*}
		Let us firstly make the additional assumption that $f_0\in F^2_{1,q}$ and $g\in L^1(0,T;F^2_{1,q})$. In this particular case, Lemmas \ref{lem:a-priori-estimates} and \ref{lem:solve-transport} insures that $f^n\in C([0,T];F^{2}_{1,q})$ and
		\begin{equation}\label{eq:equation8}
			\|f^n(t)\|_{F^2_{1,q}}\leq e^{C\int_{0}^ta(\tau)d\tau}\|f_0\|_{F^2_{1,q}}+\int^t_0 e^{C\int_{\tau}^{t}a(\tau')d\tau'}\|g(\tau)\|_{F^2_{1,q}}d\tau.
		\end{equation}
		On the other hand, we also have
		\begin{equation*}
			\|w^n(t)\|_{F^1_{1,q}}\leq \int^t_0 e^{C\int_{\tau}^{t}\|\partial_xv^n(\tau')\|_{F^1_{1,q}}d\tau'}\|(v^\infty-v^n)(\tau)\partial_xf^\infty(\tau)\|_{F^1_{1,q}}d\tau.
		\end{equation*}
		Since $F^1_{1,q}(\mathbb{R})$ is an algebra, it follows from \eqref{eq:equation8} that
		\begin{equation}\label{eq:equation9}
			\|w^n(t)\|_{F^1_{1,q}}\leq Ce^{C\int^t_0a(\tau)d\tau}\bigg(\|f_0\|_{F^2_{1,q}}+\int_{0}^t\|g(\tau)\|_{F^2_{1,q}}d\tau\bigg) \int^t_0 \|(v^\infty-v^n)(\tau)\|_{F^1_{1,q}}d\tau,
		\end{equation}
		which yield the desired result of convergence.
		
		Next, we assume that $f_0\in F^1_{1,q}$ and $g\in L^1(0,T;F^1_{1,q})$.  For all $n\in\bar{\mathbb{N}}$ and $m\in\mathbb{N}$, we introduce the solution $f^n_m$ to
		\begin{equation*}
			\begin{cases}
				\partial_tf_m^n+v^n\partial_xf_m^n=S_mg,\\
				f^n_m|_{t=0}=S_mf_0,
			\end{cases}
		\end{equation*}
		where $S_m$ is defined in \eqref{def:S_j}. Then, we have
		\begin{equation}\label{eq:equation10}
			\|w^n(t)\|_{F^1_{1,q}}\leq \|f^n-f^n_m\|_{F^1_{1,q}}+\|f^n_m-f^\infty_m\|_{F^1_{1,q}}+\|f^\infty_m-f^\infty\|_{F^1_{1,q}}.
		\end{equation}
		Obviously, $S_mf_0\in F^2_{1,q}$ and $S_mg\in L^1(0,T;F^2_{1,q})$.
		Hence, according to \eqref{eq:equation9}, we have
		\begin{equation}\label{eq:equation11}
			\|f^n_m(t)-f^\infty_m(t)\|_{F^1_{1,q}}\leq Ce^{C\int^t_0a(\tau)d\tau}\bigg(\|S_mf_0\|_{F^2_{1,q}}+\int_{0}^t\|S_mg(\tau)\|_{F^2_{1,q}}d\tau\bigg) \int^t_0 \|(v^\infty-v^n)(\tau)\|_{F^1_{1,q}}d\tau.
		\end{equation}
		On the other hand, for any $m\in\mathbb{N}$ and $k\in\bar{\mathbb{N}}$, $f^k-f^\infty_m$ solves
		\begin{equation*}
			\begin{cases}
				\partial_tu^k+v^k\partial_xu^k=g-S_mg,\\
				u^k|_{t=0}=f_0-S_mf_0.
			\end{cases}
		\end{equation*}
		Thanks to Lemma \ref{lem:a-priori-estimates} again, we obtain
		\begin{equation}
			\|f^k-f^\infty_m\|_{F^1_{1,q}}\leq e^{C\int_{0}^ta(\tau)d\tau}\bigg(\|f_0-S_mf_0\|_{F^1_{1,q}}+\int_{0}^t\|g-S_mg\|_{F^1_{1,q}}d\tau\bigg),
		\end{equation}
		which along with \eqref{eq:equation10} and \eqref{eq:equation11} yield that for all $t\in[0,T]$,
		\begin{align}
			\|w^n\|_{F^1_{1,q}}\leq& Ce^{\int_{0}^ta(\tau)d\tau}\Bigg(\|f_0-S_mf_0\|_{F^1_{1,q}}+\int_{0}^T\|g-S_mg\|_{F^1_{1,q}}d\tau\nonumber\\
			&+\bigg(\|S_mf_0\|_{F^2_{1,q}}+\int_{0}^T\|S_mg(\tau)\|_{F^2_{1,q}}d\tau\bigg) \int^T_0 \|(v^\infty-v^n)(\tau)\|_{F^1_{1,q}}d\tau\Bigg).
		\end{align}
		Note that if $q<\infty$, then for any $F^s_{p,q}$, we have $\lim\limits_{m\to\infty}\|f-S_mf\|_{F^s_{p,q}}=0.$ Hence, the Lebesgue dominated convergence theorem ensures that the first two terms of the right side can be small arbitrarily when $m$ large enough. Subsequently, if we fix some big enough $m$, then the last term tends to zero  as $n\to\infty$. Thus, we have $w^n$ tends to $0$ in $C([0,T];F^1_{1,q})$. Therefore, we have complete the proof of Lemma \ref{lem:continuity-transport}.
	\end{proof}

	\begin{proof}[Proof of Theorem \ref{thm:abstract-Local-well-posedness}]
		For the sake of convenience, we set $A(u)=u$ in the following proof.
		
		\textbf{Existence}: 
		Staring form $u^0\triangleq 0$, we then define by induction a sequence of smooth functions $(u^n)_{n\in\mathbb{N}}$ by solving the following transport equations:
		\begin{equation}\label{approximate-solution-system}
			\begin{cases}
				\partial_tu^{n+1}+u^n\partial_xu^{n+1}=F(u^n),\\
				u^{n+1}(0,x)=u_0(x).
			\end{cases}
		\end{equation}
		Using Lemma \ref{lem:solve-transport} together with \eqref{eq:existence}, we obtain a unique solution $u^{n+1}$ of \eqref{approximate-solution-system} in the space $E_q(T)$. Furthermore, applying Lemma \ref{lem:a-priori-estimates}  and \eqref{eq:existence} yields 
		\begin{equation}\label{eq:equation1}
			\|u^{n+1}(t)\|_{F^{2}_{1,q}}\leq Ce^{C\int_{0}^t\| u^n(\tau)\|_{F^2_{1,q}}d\tau}\Big(\|u_0\|_{F^{2}_{1,q}}+\int_{0}^te^{-C\int_{0}^\tau\| u^n(\tau')\|_{F^2_{1,q}}d\tau'}\big(\|u^n\|^{k+1}_{F^2_{1,q}}+1\big)d\tau\Big).
		\end{equation}
		Let us fix a $T>0$ such that $2C^2\big(\|u_0\|^{k+1}_{F^2_{1,q}}+1\big)T<1$ and suppose that
		\begin{equation}\label{eq:equation2}
			\|u^n(t)\|_{F^2_{1,q}}\leq C(\|u_0\|_{F^2_{1,q}}+1),\qquad\forall t\in[0,T],\ \forall n\in\mathbb{N}.
		\end{equation} 
		Substituting \eqref{eq:equation2} in \eqref{eq:equation1} leads to
		\begin{align*}
			\quad\|u^{n+1}(t)\|_{F^2_{1,q}}
			&\leq C\Big(\|u_0\|_{F^2_{1,q}}+\frac{C^{k+1}(\|u_0\|_{F^2_{1,q}}+1)^{k+1}+1}{2C^2(\|u_0\|^{k+1}_{F^2_{1,q}}+1)}\nonumber\Big)\leq  C(\|u_0\|_{F^2_{1,q}}+1).
		\end{align*}		
		Therefore, we obtain that the sequence $(u^n)_{n\in\mathbb{N}}$ is uniformly bounded in the space $L^\infty(0,T;F^2_{1,q})$. This clearly entails that $u^n\partial_xu^n$ is uniformly bounded in $L^\infty(0,T;F^{1}_{1,q})$, moreover, by \eqref{eq:existence}, the right-hand side of \eqref{approximate-solution-system} have been shown to be uniformly bounded in $C([0,T];F^2_{1,q})$. We can conclude that  $(u^n)_{n\in\mathbb{N}}$ is uniformly bounded in $E_{q}(T)$.

		As the embedding $F^2_{1,q} \hookrightarrow F^{1}_{1,q}$ is locally compact (by Lemma \ref{lem:Triebel-Lizorkin-properties} $(ii)$), the Arzela-Ascoli theorem and Cantor's diagonal process, we infer that, up to an extraction, $(u^n)_{n\in\mathbb{N}}$ tends to a limit $u\in \textup{Lip}([0,T];(F^{1}_{1,q})_{loc})$, i.e., for any $\phi\in C^\infty_c(\mathbb{R})$, $\phi u^n$ tends to $\phi u$ in $C([0,T];F^1_{1,q})$. Besides, by using the uniform bounds of $u^n$ and Fatou properties (Lemam \ref{lem:Triebel-Lizorkin-properties} $iv$), we gather that $u\in L^\infty(0,T;F^2_{1,q})$. By interpolation (Lemma \ref{lem:Triebel-Lizorkin-properties} $(v)$), one can deduce that $\phi u^n$ tends to $\phi u$ in $C([0,T];F^{s'}_{1,q})$ for any $s'<2$ and $\phi \in C^\infty_c$. Since $F(u)$ is a \textquoteleft good operator\textquoteright, it is a routine process to prove that  $u$ is indeed a solution to equation \eqref{general-abstract-equation}. Thanks to $u\in L^\infty(0,T;F^2_{1,q})$, the right-hand side of equation \eqref{general-abstract-equation} and Lemma \ref{lem:solve-transport}, we have $u\in C([0,T];F^2_{1,q})$. Using equations \eqref{general-abstract-equation} itself again, we can easily infer that $\partial_tu\in C([0,T];F^{1}_{1,q})$. Therefore, the obtained solution $u\in E_q(T)$.\\
		
		\textbf{Uniqueness:} 
		Since $u\in(C[0,T];F^2_{1,q}(\mathbb{R}))\hookrightarrow C([0,T];W^{1,\infty}(\mathbb{R}))$, the equation \eqref{ODE} has a unique solution $X(t,\alpha)\in C^1([0,T];C^0)$. Set $U(t,\alpha)=u(t,X(t,\alpha))$, then $U(t,\alpha)$ is bounded in $L^\infty(0,T;W^{1,\infty})$ and $U_{\alpha}(t,\alpha)=u_x(t,X(t,\alpha))X_\alpha(t,\alpha)$. In view of  \eqref{general-abstract-equation}, we obtain
		\begin{equation}\label{eq:equation12}
			X(t,\alpha)=\alpha+\int^t_0Ud\tau,\qquad\qquad\qquad X_{\alpha}(t,\alpha)=1+\int^t_0U_\alpha d\tau,
		\end{equation}
		\begin{equation}\label{eq:equation13}
			U_t(t,\alpha)=\widetilde{G}(U,X),\qquad\qquad\qquad U_{t\alpha}(t,\alpha)=\big(\widetilde{G}(U,X)\big)_{\alpha}.
		\end{equation}
		From \eqref{eq:equation12} we infer that $\frac{1}{2} \leq X_{\alpha} \leq C_{u_{0}}$ for $T>0$ sufficiently small. 
		Consequently, the flow map $X(t,\alpha)$ is one-order diffeomorphism. 
		One can easily get $U(t,\alpha) \in L^{\infty}(0,T; W^{1,1})$. Indeed,
		\begin{align*}
			\|U\|_{L^{1}}
			&= \int_{-\infty}^{+\infty} |U(t,\alpha)| d\alpha
			= \int_{-\infty}^{+\infty} |u(t, X(t,\alpha))| \frac{1}{X_{\alpha}} dX
			\leq 2 \|u\|_{L^{1}} \leq C,\\
			\|U_{\alpha}\|_{L^{1}}
			&= \int_{-\infty}^{+\infty} |U_{\alpha}(t,\alpha)| d\alpha
			= \int_{-\infty}^{+\infty} |u_{x}(t, X(t,\alpha))| dX
			\leq \|u_{x}\|_{L^{1}} \leq C.
		\end{align*}
		‌Thus, we obtain $U(t,\alpha) \in L^{\infty}(0,T; W^{1,1}\cap W^{1,\infty}),X(t,\alpha)-\alpha \in L^{\infty}(0,T; W^{1,1} \cap W^{1,\infty})$
		and $\frac{1}{2} \le X_{\alpha}(t,\alpha) \le C_{u_{0}}$ for every $t \in [0,T]$.
		
		We now establish uniqueness. Suppose $u_1$ and $u_2$ are two solutions of \eqref{general-abstract-equation}, then $U_i(t,\alpha)\triangleq u_i(t, X_i(t,\alpha))$ satisfies \eqref{eq:equation13} for $i = 1,2$. 
		Combining hypothesis \eqref{eq:uniqueness} with \eqref{eq:equation13} 
		and applying Gronwall's inequality, one can deduce
		\begin{align*}
			&\quad\|U_1-U_2\|_{W^{1,\infty} \cap W^{1,1}}+\|X_1-X_2\|_{W^{1,\infty} \cap W^{1,1}} \\
			&\leq C\big(\|U_1(0) - U_2(0)\|_{W^{1,\infty}\cap W^{1,1}}+\|X_1(0)- X_2(0)\|_{W^{1,\infty} \cap W^{1,1}}\big) \\
			&\quad + C \int_{0}^{T} \Big(
			\| \widetilde{G}(U_1, X_1)-\widetilde{G}(U_2, X_2) \|_{W^{1,\infty} \cap W^{1,1}}+\|U_{1} - U_{2}\|_{W^{1,\infty}\cap W^{1,1}} \Big)dt \\
			&\leq C \|U_1(0) - U_2(0)\|_{W^{1,\infty} \cap W^{1,1}}+C \int_{0}^{T} \Big( \|U_1-U_2\|_{W^{1,\infty} \cap W^{1,1}}+\|X_1- X_2\|_{W^{1,\infty} \cap W^{1,1}} 
			\Big) \, dt \\[4pt]
			&\leq C \|u_{1}(0) - u_{2}(0)\|_{F^{2}_{1,q}}.
		\end{align*}
		It follows that
		\begin{align*}
			\|u_1-u_2\|_{L^1}
			&\leq C \|u_1(t,X_1(t,\alpha))-u_2(t,X_1(t,\alpha))\|_{L^1} \\
			&\leq C \| 
			(u_1(t,X_1(t,\alpha))-u_2(t,X_2(t,\alpha))+u_2(t,X_2(t,\alpha))-u_2(t,X_1(t,\alpha))\|_{L^1} \\
			&\leq C \big(\|U_1-U_2\|_{L^1}+\|\partial_xu_2\|_{L^{\infty}}\|X_1-X_2\|_{L^1} \big)\\
			&\leq C \|u_{1}(0) - u_{2}(0)\|_{F^{2}_{1,q}}.
		\end{align*}
		By the embedding $L^1(\mathbb{R}) \hookrightarrow F^{-\frac{1}{2}}_{1,q}(\mathbb{R})$, one gets
		\[
		\|u_1- u_2\|_{F^{-\frac{1}{2}}_{1,q}}\leq C\|u_1-u_2\|_{L^1}\leq C \|u_1(0)- u_2(0)\|_{F^{2}_{1,q}}.
		\]
		Thus, if $u_{1}(0) = u_{2}(0)$, uniqueness follows immediately.\\
		
		\textbf{Continuous dependence:} Assume that $u^{n}_{0} \to u^{\infty}_{0}$ in $F^{2}_{1,q}$, and let $u^n$, $u^{\infty}$ be the corresponding solutions of \eqref{general-abstract-equation} with a common lifespan $T$. According to the previous discussion, we see that $u^{n}$ and $u^{\infty}$ are uniformly bounded in $L^{\infty}(0,T; F^{2}_{1,q})$, and
		\begin{equation}\label{eq:equation14}
			\| u^{n}-u^{\infty}\|_{F^{-\frac{1}{2}}_{1,q}}
			\leq C \|u^{n}_{0}-u^{\infty}_{0}\|_{F^{2}_{1,q}},
			\qquad \forall\, t \in [0,T].
		\end{equation}
		By the interpolation inequality, we infer that $u^{n} \to u^{\infty}$ in $C([0,T]; F^{1}_{1,q})$. 
		
		In order to prove $u^{n} \to u^{\infty}$ in $C([0,T]; F^{2}_{1,q})$, it suffices to show that $u^n_x \to u^\infty_x$ in $C([0,T]; F^{1}_{1,q})$. 
		For brevity, let $v^n\triangleq u^n_x$ and $v^\infty\triangleq u^{\infty}_{x}$. We decompose $v^{n}=w^{n}+z^{n}$ with $(w^{n}, z^{n})$ satisfying
		\[
		\begin{cases}
			\partial_{t} w^{n} + u^{n} \partial_{x} w^{n}
			= \partial_{x}(G(u^{\infty})) - (u^{\infty}_{x})^{2}, \\
			w^{n}(0,x) = v^{n}_{0} = \partial_{x} u^{n}_{0},
		\end{cases}
		\]
		and
		\[
		\begin{cases}
			\partial_{t} z^{n} + u^{n} \partial_{x} z^{n}
			= \partial_{x}(G(u^{n}) - G(u^{\infty}))
			-((u^{n}_{x})^{2} - (u^{\infty}_{x})^{2}), \\
			z^{n}(0,x) = v^{n}_{0}-v^{\infty}_{0}
			= \partial_{x} u^{n}_{0}-\partial_{x} u^{\infty}_{0}.
		\end{cases}
		\]
	   Then \eqref{eq:existence} together with Lemma \ref{lem:continuity-transport} guarantees that
		\begin{equation}\label{eq:equation15}
			w^{n} \to w^{\infty}\quad\text{in}\quad C([0,T]; F^{1}_{1,q}).
		\end{equation}
		By virtue of \eqref{eq:continuous} and Lemma \ref{lem:Triebel-Lizorkin-properties} $(vi)$, we have
		\begin{align*}
			\| \partial_x(G(u^{n})-G(u^{\infty})) \|_{F^{1}_{1,q}}
			\leq C \|u^n-u^\infty\|_{F^{2}_{1,q}} \leq C\big(
			\|u^{n} - u^{\infty}\|_{F^{1}_{1,q}}+\|u^n_x- u^{\infty}_{x}\|_{F^{1}_{1,q}}\big),
		\end{align*}
		and
		\begin{align*}
			\| (u^{n}_{x} + u^{\infty}_{x})(u^{n}_{x} - u^{\infty}_{x}) \|_{F^{1}_{1,q}}
			\leq C \|u^{n}_{x} - u^{\infty}_{x}\|_{F^{1}_{1,q}}.
		\end{align*}
		From Lemma \ref{lem:a-priori-estimates} and $u^n$ is uniformly bounded in $L^\infty(0,T;F^2_{1,q})$ it follows that, for all $n \in \mathbb{N}$,
		\begin{align}
			\|z^{n}(t)\|_{F^{1}_{1,q}}
			&\leq C \Big(
			\|v^{n}_{0} - v^{\infty}_{0}\|_{F^{1}_{1,q}}+ \int_{0}^{t}
			\|u^{n}-u^{\infty}\|_{F^{1}_{1,q}}+\|u^{n}_{x}-u^{\infty}_{x}\|_{F^{1}_{1,q}} d\tau\Big) \nonumber\\
			&\leq C \Big(
			\|v^{n}_{0}-v^{\infty}_{0}\|_{F^{1}_{1,q}}+\int_{0}^{t}\|u^{n}-u^{\infty}\|_{F^{1}_{1,q}}+\|w^{n}-w^{\infty}\|_{F^{1}_{1,q}}+\|z^{n}\|_{F^{1}_{1,q}}d\tau\Big).\label{eq:equation16}
		\end{align}
		Now we invoke the following facts:
		\begin{itemize}
			\item $v^{n}_{0} \to v^{\infty}_{0}$ in $F^{1}_{1,q}$,
			\item $u^{n} \to u^{\infty}$ in $C([0,T]; F^{1}_{1,q})$,
			\item $w^{n} \to w^{\infty}$ in $C([0,T]; F^{1}_{1,q})$.
		\end{itemize}
		Applying Gronwall's inequality to \eqref{eq:equation16}, one infers that
		$z^{n}$ tends to $0$ in $C([0,T]; F^{1}_{1,q})$. In view of Lemmas \ref{lem:a-priori-estimates}-\ref{lem:solve-transport}, we also have $z^{\infty} = 0$ in $C([0,T]; F^{1}_{1,q})$. Therefore,
		\begin{align*}
			\|v^n-v^\infty\|_{L^{\infty}(0,T; F^{1}_{1,q})}
			&\leq \|w^n-w^\infty\|_{L^{\infty}(0,T; F^{1}_{1,q})}+\|z^n-z^\infty\|_{L^{\infty}(0,T;F^{1}_{1,q})} \\
			&\leq \|w^n-w^\infty\|_{L^{\infty}(0,T; F^{1}_{1,q})}+\|z^{n}\|_{L^{\infty}(0,T; F^{1}_{1,q})}\longrightarrow 0 \quad\text{as } n \to \infty,
		\end{align*}
		that is,
		\[
		u^n_x \to u^\infty_x\quad\text{in}\quad C([0,T]; F^{1}_{1,q}).
		\]
		Thus we have proved the continuous dependence of the solution map for \eqref{general-abstract-equation} in $C([0,T]; F^{2}_{1,q})$ with $1 \leq q < \infty$. Therefore, we complete the proof of Theorem \ref{thm:abstract-Local-well-posedness}
	\end{proof}

	\begin{proof}[Proof of Theorem \ref{thm:CH-wellposedness}]
		Let $A(u)=u$ and $G(u)=-\partial_x(1-\partial^2_x)^{-1}\big(u^2+\frac{1}{2}(\partial_xu)^2\big)$. Then \eqref{general-abstract-equation} becomes the Camassa-Holm equation \eqref{CH-transport-equation}. Using approximation argument, it is easy to see that $G(u)$ is a \textquoteleft good operator\textquoteright. Hence, by Theorem \ref{thm:abstract-Local-well-posedness}, it suffices to verify conditions \eqref{eq:existence}-\eqref{eq:continuous}. The proof of \eqref{eq:uniqueness} (which is independent of the working space $F^s_{p,q}$) can be found in \cite{MR4595933}, we only need to prove \eqref{eq:existence} and \eqref{eq:continuous}. Indeed,
		\begin{equation*}
			\|G(u)\|_{F^2_{1,q}}=\|\partial_x(1-\partial^2_x)^{-1}(u^2+\frac{1}{2}u_x^2)\|_{F^2_{1,q}}\leq C(\|u\|^2_{F^2_{1,q}}+1),
		\end{equation*}
		\begin{equation*}
			\|G(u_1)-G(u_2)\|_{F^2_{1,q}}\leq C(\|u_1\|_{F^2_{1,q}}+\|u_2\|_{F^2_{1,q}})(\|u_1-u_2\|_{F^2_{1,q}})\leq  C(\|u_1-u_2\|_{F^2_{1,q}}).
		\end{equation*}
		Therefore, we complete the proof of Theorem \ref{thm:CH-wellposedness}.
	\end{proof}
	
	\section{Ill-posedness in the sense of norm inflation}
	The aim of this section is to show the ill-posedness in $F^{1+\frac{1}{p}}_{p,q}(\mathbb{R})$ for $1<p<\infty$ and $1\leq q \leq \infty$.  To this end, we first introduce the atomic decomposition in Triebel-Lizorkin spaces. For $\nu\in\mathbb{N}$ and $m\in\mathbb{Z}$, let $Q_{\nu m}$ be the closed interval centred at $2^{-\nu}m$ with length $2^{-\nu+1}$ (i.e., $Q_{\nu m}=[2^{-\nu}m-2^{-\nu},2^{-\nu}m+2^{-\nu}]$). If $Q$ is an interval and $r>0$, then $rQ$ denotes the interval concentric with $Q$ whose length is $r$ times the length of $Q$.
	
	\begin{definition}\label{def:atom}
		\textup{(\cite{MR2250142})}
		\begin{enumerate}[label={(\roman*)}]
			\item Let $K\in\mathbb{N}$ and $c\geq 1$. A continuous function $a: \mathbb{R} \to \mathbb{C}$ for which there exist all derivatives $D^\alpha a$ if $|\alpha|\leq K$ is called a $1$-atom (more precisely $1_K$-atom) if 
			\[ \text{supp}\, a \subset c\,Q_{0m}\qquad \text{for some }\quad m\in\mathbb{Z}\]
			and 
			\[ |D^\alpha a(x)|\leq 1\qquad\text{for}\qquad|\alpha|\leq K.  \]
			\item Let $s\in\mathbb{R},1\leq p\leq \infty,K\in\mathbb{N},L\geq 0$ and $c\geq 1$. A continuous function $a: \mathbb{R} \to \mathbb{C}$ for which there exist all derivatives $D^\alpha a$ if $|\alpha|\leq K$ is called a $(s,p)$-atom (more precisely $(s,p)_{K,L}$-atom) if
			\[ \text{supp}\, a \subset c\,Q_{\nu m}\qquad \text{for some }\quad\nu\in\mathbb{N}^+,\quad m\in\mathbb{Z},\]
			\[ |D^\alpha a(x)|\leq 2^{-\nu(s-\frac{1}{p})+|\alpha|\nu}\qquad\text{for}\qquad|\alpha|\leq K, \]
			and 
			\[ \int_{\mathbb{R}}x^\beta a(x)dx=0\qquad\text{for}\qquad |\beta|<L. \]
		\end{enumerate}
	\end{definition}
	\begin{definition}\label{def:f}
		\textup{(\cite{MR2250142})} Let $1\leq p< \infty,1\leq q\leq \infty$. Set $\lambda\triangleq\{\lambda_{\nu m}\in\mathbb{C}:\nu\in\mathbb{N},m\in\mathbb{Z}\}$,
		and 
		\[ \|\lambda\|_{f_{pq}}\triangleq\Big\|\big(\sum_{\nu=0}^{\infty}\sum_{m\in \mathbb{Z}}|\lambda_{\nu m}\chi^{(p)}_{\nu m}(\cdot)|^q\big)^{\frac{1}{q}}\Big\|_{L^p(\mathbb{R})}, \]
		where 
		\begin{equation}\label{eq:equation19}
			\chi^{(p)}_{\nu m}(x)=
			\begin{cases}
				2^{(\nu-1)/p},\quad&\text{if}\quad x\in Q_{\nu m},\\
				0,\quad&\text{if}\quad x\notin Q_{\nu m}
			\end{cases}
		\end{equation}
		with the usual modification if $q=\infty$. Moreover, we define $f_{pq}\triangleq\{\lambda:\|\lambda\|_{f_{pq}}<\infty\}.  $
	\end{definition}
	
	\begin{lemma}\label{lem:F-atoms}
		\textup{(\cite{MR2250142})} Let $1\leq p<\infty,1\leq q\leq \infty,s\in\mathbb{R}$. Let $ K\in \mathbb{N},L\geq 0$ with 
		\begin{equation}\label{eq:equation17}
			K>s\qquad\text{and}\qquad L>-s
		\end{equation}
		be fixed. Then $f\in F^s_{p,q}(\mathbb{R})$ if and only if it can be represented as 
		\begin{equation}\label{eq:equation18}
			f=\sum_{\nu=0}^\infty\sum_{m\in \mathbb{Z}}\lambda_{\nu m}a_{\nu m},\qquad \text{in } \mathscr{S}'(\mathbb{R}),
		\end{equation}
		where for fixed $c\geq 1$, $a_{\nu m}$ are $1_K$-atoms ($\nu=0$) or $(s,p)_{K,L}$-atoms ($\nu\in\mathbb{N}^+$) with respect to \eqref{eq:equation17} and $\lambda\in f_{pq}$, Furthermore,
		\[ \|f\|_{F^s_{p,q}(\mathbb{R})}\sim\inf\|\lambda\|_{f_{pq}} \]
		are equivalent norms where the infimum is taken over all admissible representations \eqref{eq:equation18}.
	\end{lemma}
	
	Let 
	\begin{equation}\label{eq:equation20}
		E=\{E_{\nu m}:\nu\in\mathbb{N},m\in\mathbb{Z}\}, \quad E_{\nu m}\subset Q_{\nu m},\quad |E_{\nu m}|\sim|Q_{\nu m}|, 
	\end{equation}
	where the equivalence constants are independent of $\nu$ and $m$. Let $\chi_{\nu m}^{(p),E}$ be given by \eqref{eq:equation19} with $E_{\nu m}$ in place of $Q_{\nu m}.$ 

	\begin{lemma}\label{lem:atom-1}
		\textup{(\cite{MR2250142})} Let $1\leq p<\infty $, $1\leq q\leq \infty$ and $\lambda\triangleq\{\lambda_{\nu m}\in\mathbb{C}:\nu\in\mathbb{N},m\in\mathbb{Z}\}$. Let $E$ be given by \eqref{eq:equation20}, then 
		\[ f_{pq}=f^{E}_{pq}\triangleq\{\lambda:\|\lambda\|_{f_{pq}^E}<\infty\}, \]
		where 
		\[ \|\lambda\|_{f_{pq}^E}=\Big\|\big(\sum_{\nu=0}^{\infty}\sum_{m\in \mathbb{Z}}|\lambda_{\nu m}\chi^{(p),E}_{\nu m}(\cdot)|^q\big)^{\frac{1}{q}}\Big\|_{L^p(\mathbb{R})}  \]
		is an equivalent norm depending on the equivalence constants in \eqref{eq:equation20}.
	\end{lemma}
	
	Next, we recall several results about the solutions to the CH equation. 
	
	\begin{lemma}\label{lem:CH-jie}
		\textup{(\cite{MR1741872,MR1851854,zhang2026commutatorestimatesapplicationstransporttype})}. Given $u_0\in H^s(\mathbb{R}),s>\frac{3}{2}$, there exists a maximal $T\geq \tilde{T}(\|u_0\|_{H^s})>0$ and a unique solution $u$ to \eqref{CH-transport-equation} such that 
		\[ u\in C([0,T);H^s)\cap C^1([0,T);H^{s-1}). \]
		Moreover, the solution depends continuously on the initial data, i.e. the mapping $u_0\mapsto u:H^s\to C([0,T);H^s)\cap C^1([0,T);H^{s-1})$ is continuous on a $H^s$-neighborhood of $u_0$, and if $T<\infty$, then $\lim_{t\to T^-}\|u(t)\|_{H^s}=\infty.$
	\end{lemma}
	
	\begin{lemma}\label{lem:CH-time}
		\textup{(\cite{MR1631589})}Assume $u_0\in H^3(\mathbb{R})$ is a real-valued odd function and $u'_0(0)<0$. Then the corresponding strong solution to \eqref{CH-transport-equation} blows up in finite time. Moreover, the maximal time of existence is bounded by $2/|u'_0(0)|$.
	\end{lemma}
	
	\begin{lemma}\label{lem:Lguji}
		\textup{(\cite{MR3906228})} Assume $u\in H^2(\mathbb{R})$, we have
		\[ \|u_x\|_{L^\infty}\leq C\|u\|_{F^1_{\infty,\infty}}\log_2(2+\|u\|^2_{H^2})+C. \]
	\end{lemma}
	
	\begin{proof}[Proof of Theorem \ref{thm:CH-illposedness}]
		Let $a(x)$ be an even function in $C^\infty_c(\mathbb{R})$ with support in $[-1,1]$ and values in $[-1,0)$. Set
		\[
		 M=\max_{|\alpha|\leq 2}\sup_{x\in\mathbb{R}}|D^\alpha a(x)|,\qquad\tilde{a}(x)=\frac{a(x)}{3M},\qquad\tilde{a}_k(x)=\tilde{a}(2^kx),\quad k\in\mathbb{N}.
		\]
		It is easy to see that for all $k\in\mathbb{N}$, the following estimates hold:
		\[
		\text{supp }\big(x\tilde{a}_k(x)\big)\subset Q_{k0},
		 \]
		\[\qquad |x\tilde{a}_k(x)|\leq 2^{-k},   \]
		\[ |D(x\tilde{a}_k)|=|\tilde{a}_k(x)+x2^k\tilde{a}'(2^kx)|\leq 1,\]
		\[ |D^2(x\tilde{a}_k)|=|2^k\tilde{a}'(2^kx)+2^k\tilde{a}'(2^kx)+x2^{2k}\tilde{a}''(2^kx)|\leq 2^k,
		\]
		and
		\[ \int_{\mathbb{R}}x\tilde{a}_k(x)dx=0. \]
		Hence, for $1<p<\infty,1\leq q\leq \infty$, the functions $b_{k0}(x)\triangleq x\tilde{a}_k(x)$ are $(1+\frac{1}{p},p)_{2,1}$-atoms in $F^{1+\frac{1}{p}}_{p,q}(\mathbb{R})$ for $k\in\mathbb{N}^+$ (and $b_{00}(x)$ is a $1_2$-atom), according to Definition \ref{def:atom} and admissible in Lemma \ref{lem:F-atoms} with supports in the intervals $Q_{k 0}$ centred at the origin. Let 
		\begin{equation}
			f(x)=\sum^\infty_{k=0}
			\lambda_{k0}b_{k0}(x)\triangleq \sum_{k=0}^\infty \frac{1}{k+1}b_{k0}(x).
		\end{equation}
		Then 
		\begin{equation}
			\|f\|_{F^{1+\frac{1}{p}}_{p,q}}\leq \|\lambda\|_{f_{pq}}\sim 
			\|\lambda\|_{f^{E}_{pq}}\sim \Big(\sum_{k=0}^\infty\Big|\frac{1}{k+1}\Big|^p\Big)^{\frac{1}{p}}
		\end{equation}
		with the last term independent of $q$, which follows from Lemma \ref{lem:atom-1} if one chooses the sets $E_{\nu0}$ in \eqref{eq:equation20} such that $E_{\nu0}\cap E_{\mu0}=\varnothing$ if $\nu\neq\mu$. Thus, we see that $f\in F^{1+\frac{1}{p}}_{p,q}$, and 
		\[ f'(0)=\sum_{k=0}^\infty\frac{1}{k+1}\tilde{a}_k(0)=-\infty. \]
		For $\varepsilon>0$, by \eqref{def:S_j}, we take $u_{0,\varepsilon}\triangleq\|f\|^{-1}_{F^{1+\frac{1}{p}}_{p,q}}\cdot\varepsilon S_n(f)$ with $n$ sufficiently large such that $u'_{0,\varepsilon}(0)<-2\varepsilon^{-10}$. Then $u_{0,\varepsilon}$ is a real valued odd function, $u_{0,\varepsilon}\in H^\infty$, and  $\|u_{0,\varepsilon}\|_{F^{1+\frac{1}{p}}_{p,q}}\leq \varepsilon$. By Lemmas 
		\ref{lem:CH-jie} and \ref{lem:CH-time}, there exists a unique solution $u\in C([0,T);H^\infty(\mathbb{R}))$ with a maximal lifespan $T_{\varepsilon}< \varepsilon^{10}$. In order to prove Theorem \ref{thm:CH-illposedness}, it suffices to show 
		\begin{equation}\label{eq:equation21}
			\limsup_{t\to T^-_{\varepsilon}}\|u_\varepsilon\|_{F^1_{\infty,\infty}}=\infty.
		\end{equation}
		Indeed, suppose \eqref{eq:equation21} fails, then there exists a constant $M_\varepsilon>1$ such that
		\begin{equation}\label{eq:equation21-5}
			\sup_{t\in[0,T_\varepsilon)}\|u_\varepsilon(t)\|_{F^1_{\infty,\infty}}\leq M_\varepsilon.
		\end{equation}
		For $s>\frac{3}{2}$, note that the following standard energy estimate
		\begin{equation}\label{eq:equation22}
			\frac{d}{dt}\|u\|^2_{H^s}\leq C\|u_x\|_{L^\infty}\|u\|^2_{H^s}
		\end{equation}
		holds for any solution $u$ to \eqref{CH-transport-equation}. Combining \eqref{eq:equation21-5}-\eqref{eq:equation22} with Lemma \ref{lem:Lguji} yields
		\begin{align*}
			\frac{d}{dt}\|u_\varepsilon\|^2_{H^2}
			&\leq C\|\partial_xu_{\varepsilon}\|_{L^\infty}\|u_\varepsilon\|^2_{H^2}\\
			&\leq C\big(\|u_\varepsilon\|_{F^1_{\infty,\infty}}\log_2(2+\|u_\varepsilon\|^2_{H^2})+1\big)\|u_\varepsilon\|^2_{H^2}\\
			&\leq CM_\varepsilon\|u_\varepsilon\|^2_{H^2}\log_2(2+\|u_\varepsilon\|^2_{H^2})
		\end{align*}
		A straightforward computation gives  $\sup\limits_{t\in[0,T_\varepsilon)}\|u_\varepsilon\|^2_{H^2}<\infty$, which contradicts to the blow-up criterion in Lemma \ref{lem:CH-jie}. Therefore, we complete the proof of Theorem \ref{thm:CH-illposedness}.
	\end{proof}

	\noindent{\bf Acknowledgments.} This work was partially supported by the National Natural Science Foundation of China under grant 11971188.
	
	
	
	
	

	\bibliographystyle{abbrv}
	\bibliography{ref}
	
\end{document}